\documentclass[a4paper,11pt]{article}

\usepackage{a4wide}
\usepackage[english]{babel}
\usepackage{amsfonts}
\usepackage{amsmath}
\usepackage{graphicx}
\usepackage[colorlinks,bookmarks=true]{hyperref}
\usepackage{amsthm}
\usepackage{enumitem}

\newtheorem{lemma}{Lemma}[section]

\newtheorem{theorem}{Theorem}[section]
\newtheorem{proposition}{Proposition}[section]

\newtheorem{corollary}{Corollary}[section]

\newcommand{\E}{\mathbb{E}}

\def\be{\begin{eqnarray}}
\def\ee{\end{eqnarray}}
\def\b*{\begin{eqnarray*}}
\def\eq*{\end{eqnarray*}}
\def\beq{\begin{equation}}
\def\eeq{\end{equation}}

\title{Some explicit formulas for the Brownian bridge, Brownian meander and Bessel process under uniform sampling}

\author{Mathieu Rosenbaum and Marc Yor\\$~~$\\
LPMA, University Pierre et Marie Curie (Paris 6)}

\begin{document}

\maketitle

\begin{abstract}
\noindent We show that simple explicit formulas can be obtained for several relevant quantities related to the laws of the uniformly sampled Brownian bridge, Brownian meander and three dimensional Bessel process. To prove such results, we use the distribution of a triplet of random variables associated to the pseudo-Brownian bridge given in \cite{rosenbaum2013law}, together with various relationships between the laws of these four processes. Finally, we consider the variable $B_{UT_1}/\sqrt{T_1}$, where $B$ is a Brownian motion, $T_1$ its first hitting time of level one and $U$ a uniform random variable independent of $B$. This variable is shown to be centered in \cite{elie2013expectation,rosenbaum2013law}. The results obtained here enable us to revisit this intriguing property through an enlargement of filtration formula.  
\end{abstract}


\section{Introduction}\label{intro}

Let $(B_t,~t\geq 0)$ be a standard Brownian motion, $T_1$ its first hitting time of level one, and
$U$ a uniform random variable on $[0,1]$, independent of $B$. In \cite{elie2013expectation}, it is first shown that the random variable $\alpha$ defined by
\begin{equation}\label{variable}
\alpha=\frac{B_{UT_1}}{\sqrt{T_1}}
\end{equation}
is centered. Intrigued by this property, we determined the distribution of this variable, which is expressed in \cite{rosenbaum2013law} under the following form, where $\underset{\mathcal{L}}{=}$ denotes equality in law:
\begin{equation}\label{law}
\alpha\underset{\mathcal{L}}{=}\Lambda L_1-\frac{1}{2}|B_1|,\end{equation}
with $L_t$ the local time at point $0$ of $B$ at time $t$ and $\Lambda$ a uniform random variable on $[0,1]$, independent of $(|B_1|,L_1)$. The centering property is easily recovered from \eqref{law} since
$$\E[\Lambda L_1-\frac{1}{2}|B_1|]=\frac{1}{2}\E[L_1-|B_1|]=0.$$

\noindent In fact, in \cite{rosenbaum2013law}, a preliminary to the proof of \eqref{law} is to obtain the law of a triplet of random variables defined in terms of the pseudo-Brownian bridge introduced in \cite{biane1987processus}, see Section \ref{sec_prelim} below. In this paper, we show that the law of this triplet enables us to derive several unexpected simple formulas for various quantities related to some very classical Brownian type processes, namely the Brownian bridge, the Brownian meander and the three dimensional Bessel process. More precisely, we focus on distributional properties of these processes when sampled with an independent uniform random variable. Thus, this work can be viewed as a modest complement to the seminal paper by Pitman, see \cite{pitman1999brownian}, where the laws of these processes when sampled with (several) independent uniform random variables are already studied.\\

\noindent The paper is organized as follows. In Section \ref{sec_prelim}, we give some preliminary results related to the law of $(|B_1|,L_1)$. Indeed, they play an important role in the proofs. Distributional properties for the Brownian bridge are established in Section \ref{sec_bri} whereas the Brownian meander and the three dimensional Bessel process are investigated in Section \ref{sec_mea}. Finally, in Section \ref{sec_filt}, we reinterpret the fact that $\alpha$ is centered through the lenses of an enlargement formula for the Brownian motion with the time $T_1$ due to Jeulin, see \cite{jeulin1979grossissement}. In particular we show that this centering property can be translated in terms of the expectation of the random variable 
$U/(R_UR_1^2)$, where $R$ is a three dimensional Bessel process and $U$ a uniform random variable on $[0,1]$ independent of $R$.

\section{Some preliminary results about the law of $(|B_1|,L_1)$}\label{sec_prelim}

Before dealing with the Brownian bridge, the Brownian meander and the three dimensional Bessel process, we give here some preliminary
results related to the distribution of the couple $(|B_1|,L_1)$. These results will play an important role in the proofs of our main theorems.\\

\noindent It is well known that the law of $(B_1,L_1)$ admits a density on $\mathbb{R}\times\mathbb{R}^+$. Its value at point $(x,l)$ is given by
\begin{equation}\label{dens}
\frac{1}{\sqrt{2\pi}}(|x|+l)\text{exp}\big(-\frac{(l+|x|)^2}{2}\big).
\end{equation}
For $l\geq 0$, we set
$$H(l)=\text{e}^{l^2/2}\int_l^{+\infty}dx\text{e}^{-x^2/2}.$$
The following consequences of \eqref{dens} shall be useful in the sequel.

\begin{proposition}\label{prop_prelim}
Let $l\geq 0$. We have the double identity
\begin{equation}\label{double}
\E[L_1||B_1|=l]=\E[|B_1||L_1=l]=H(l).
\end{equation}
Furthermore, one has
\begin{equation}\label{double2}
H(l)=l\E[\frac{1}{N^2+l^2}],\end{equation}
where $N$ denotes a standard Gaussian random variable.
\end{proposition}

\begin{proof}
We start with the proof of \eqref{double}. Of course, it can be deduced from \eqref{dens} at the cost of some integrations. We
prefer the following arguments. First, the equality on the left hand side of \eqref{double} stems from the symmetry of the law of 
$(|B_1|,L_1)$, which is obvious from \eqref{dens}. Thus, we now have to show the second equality in \eqref{double}. This easily follows from the identity
\begin{equation}\label{balayage}
\E[\phi(L_1)|B_1|]=\E[\int_0^{L_1}dx\phi(x)],
\end{equation}
which is valid for any bounded measurable function $\phi$. Indeed, assuming \eqref{balayage} for a moment, using the fact that 
$$L_1\underset{\mathcal{L}}{=}|B_1|,$$
we may write \eqref{balayage} as
$$\int_0^{+\infty}dl\text{e}^{-l^2/2}\phi(l)\E[|B_1||L_1=l]=\int_0^{+\infty}dl\phi(l)\int_l^{+\infty}dx\text{e}^{-x^2/2}.$$
Hence, since this is true for every bounded measurable function $\phi$, we get
$$\text{e}^{-l^2/2}\E[|B_1||L_1=l]=\int_l^{+\infty}dx\text{e}^{-x^2/2},$$
which is the desired result for \eqref{double}.\\

\noindent It remains to prove \eqref{balayage} for a generic bounded measurable function $\phi$. Remark that the formula
$$\phi(L_t)|B_t|=\int_0^tdB_s\phi(L_s)\text{sign}(B_s)+\int_0^tdL_s\phi(L_s)$$
is a very particular case of the balayage formula, see \cite{revuz1999continuous}, page 261. It now suffices to take expectation on both sides of this last equality to obtain \eqref{balayage}.\\

\noindent We now give the proof of the second part of Proposition \ref{prop_prelim}. First, note that
$$\E[\frac{1}{N^2+l^2}]=\int_0^{+\infty}dv\text{e}^{-vl^2}\E[\text{e}^{-vN^2}].$$
Using the Laplace transform of $N^2$, we obtain 
\begin{equation}\label{form}
\E[\frac{1}{N^2+l^2}]=\int_0^{+\infty}dv\frac{\text{e}^{-vl^2/2}}{2\sqrt{1+v}}.
\end{equation}
Then remark that thanks to the change of variable $x^2=(1+v)l^2$, we get
$$H(l)=\text{e}^{l^2/2}\int_l^{+\infty}dx\text{e}^{-x^2/2}=\int_0^{+\infty}dv\frac{l}{2\sqrt{1+v}}\text{e}^{-vl^2/2}.$$
This together with $\eqref{form}$ gives the second part of Proposition \ref{prop_prelim}.
\end{proof}

\section{The Brownian bridge under uniform sampling}\label{sec_bri}

Before giving our theorem on the uniformly sampled Brownian bridge, we recall a result related to the pseudo-Brownian bridge established in \cite{rosenbaum2013law}, and which is the key to most of the proofs in this paper. The pseudo-Brownian bridge was introduced in \cite{biane1987processus} and is defined by
$$(\frac{B_{u\tau_1}}{\sqrt{\tau_1}},~u\leq 1),$$ with
$(\tau_l,~l>0)$ the inverse local time process:
$$\tau_l=\text{inf}\{t,~L_t>l\}.$$
This pseudo-Brownian bridge is equal to $0$ at time $0$ and time $1$ and has the same quadratic variation as the Brownian motion. Thus, it shares some similarities with the Brownian bridge, which explains its name. Let $U$ be a uniform random variable on $[0,1]$ independent of $B$. The following theorem is proved in \cite{rosenbaum2013law}.
\begin{theorem}\label{theo1}
There is the identity in law
$$(\frac{B_{U\tau_1}}{\sqrt{\tau_1}},\frac{1}{\sqrt{\tau_1}},L_{U\tau_1})\underset{\mathcal{L}}{=}(\frac{1}{2}B_1,L_1,\Lambda),$$
with $\Lambda$ a uniform random variable on $[0,1]$, independent of $(B_1,L_1)$.
\end{theorem}
\noindent In other words, $L_{U\tau_1}$ is a uniform random variable on $[0,1]$, independent of the pair $$(\frac{B_{U\tau_1}}{\sqrt{\tau_1}},\frac{1}{\sqrt{\tau_1}}),$$ which is distributed as $(\frac{1}{2}B_1,L_1)$.\\

\noindent To deduce some properties of the Brownian bridge from Theorem \ref{theo1}, the idea is to use an absolute continuity relationship between the law of the pseudo-Brownian bridge and that of the Brownian bridge shown by Biane, Le Gall and Yor in \cite{biane1987processus}. More precisely, for $F$ a non negative measurable function on $\mathbb{C}([0,1],\mathbb{R})$, we have
\begin{equation}\label{bly}
\E\big[F\big(\frac{B_{u\tau_1}}{\sqrt{\tau_1}},~u\leq 1\big)\big]=\sqrt{\frac{2}{\pi}}\E\big[F\big(b(u),~u\leq 1\big)\frac{1}{\lambda_1^0}\big],\end{equation}
where $\big(b(u),~u\leq 1\big)$ denotes the Brownian bridge and $(\lambda_u^x,~u\leq 1,~x\in \mathbb{R})$ its family of local times. 
Let $U$ be again a uniform random variable on $[0,1]$ independent of $b$. The following theorem is easily deduced from Theorem \ref{theo1} together with Equation \ref{bly}.

\begin{theorem}\label{theo2}
For any non negative measurable functions $f$ and $g$, we have
\begin{equation}\label{pont}
\E\big[f\big(b(U),\lambda_1^0\big)g\big(\frac{\lambda_U^0}{\lambda_1^0}\big)\big]=\sqrt{\frac{\pi}{2}}\E\big[f(\frac{1}{2}B_1,L_1)L_1\big]\E[g(\Lambda)],\end{equation}
with $\Lambda$ a uniform random variable on $[0,1]$, independent of $(B_1,L_1)$.
\end{theorem}
\noindent Thus, $\lambda_U^0/\lambda_1^0$ is a uniform random variable on $[0,1]$, independent of the pair $\big(b(U),\lambda_1^0\big)$ which is distributed according to \eqref{pont} with $g=1$.\\

\noindent The following corollary of Theorem \ref{theo2} provides some surprisingly simple expressions for some densities and (conditional) expectations of quantities related to the Brownian bridge.

\begin{corollary}\label{cortheo2}
The following properties hold:\\

\noindent $\bullet~$The variable $\lambda_1^0$ admits a density on $\mathbb{R}^+$. Its value at point $l\geq 0$ is given by
$$l\emph{exp}(-l^2/2).$$ 
Hence, $\lambda_1^0$ has the same law as $\sqrt{2\mathcal{E}}$, with $\mathcal{E}$ an exponential random variable. Therefore,
$\lambda_1^0$ is Rayleigh distributed.\\

\noindent $\bullet~$The density of $b(U)$ at point $y$ given $\lambda_1^0=l$ is given by
$$\E[\lambda_1^y|\lambda_1^0=l]=(2|y|+l)\emph{exp}\big(-(2y^2+2|y|l)\big).$$
Consequently, there is the formula
$$\E\big[\frac{\lambda_1^y}{\lambda_1^0}\big]=\emph{exp}(-2y^2).$$

\noindent $\bullet~$The density of $b(U)$ at point $y$ is given by
$$\E[\lambda_1^y]=\int_{2|y|}^{+\infty}dz\emph{exp}(-z^2/2).$$
Thus, we have $b(U)\underset{\mathcal{L}}{=}\sqrt{2\mathcal{E}}(V/2)$, with $\mathcal{E}$ an exponential variable independent of $V$ which is uniformly distributed on $[-1,1]$.
\end{corollary}
\noindent The first part of Corollary \ref{cortheo2} is obviously deduced from Theorem \ref{theo2}
and is in fact a very classical result, see \cite{biane1987processus,biane1988quelques,imhof1984density,revuz1999continuous}. We now prove the second part.
\begin{proof}
\noindent Let $f$ be a non negative measurable function. First note that
$$\E\big[f\big(b(U)\big)|\lambda_1^0=l\big]=\E\big[\int_0^1 du f\big(b(u)\big)|\lambda_1^0=l\big]=\int_{\mathbb{R}} dy f(y)\E[\lambda_1^y|\lambda_1^0=l].$$
Hence the density of $b(U)$ at point $y$ given $\lambda_1^0=l$ is equal to 
$$\E[\lambda_1^y|\lambda_1^0=l].$$
Now, let $h$ denote the density of the couple $(B_1,L_1)$ given in Equation \eqref{dens}. From Theorem \ref{theo2}, we easily get that the density of $b(U)$ at point $y$ given $\lambda_1^0=l$ is equal to
$$2\sqrt{\frac{\pi}{2}}\frac{lh(2y,l)}{l\text{exp}(-l^2/2)}=\sqrt{2\pi}h(2y,l)\text{exp}(l^2/2).$$
The first statement in the second part of Corollary \ref{cortheo2} readily follows from Equation \eqref{dens}. For the second statement, we use the fact that
$$
\E\big[\frac{\lambda_1^y}{\lambda_1^0}\big]=\E\big[\int_0^{+\infty}dl\frac{1}{l}\E[\lambda_1^y|\lambda_1^0=l]\big]l\text{exp}(-l^2/2)=\sqrt{2\pi}\int_0^{+\infty}dlh(2|y|,l).
$$
Using the definition of $h$, this last expression is equal to 
$$\text{exp}(-2y^2).$$
\end{proof}
\noindent The last identity in Corollary \ref{cortheo2} is easily deduced from Theorem \ref{theo2} together with Proposition \ref{prop_prelim}. Note that this formula can also be found in \cite{shorack2009empirical}, page 400.

\section{The Brownian meander and the three dimensional Bessel process under uniform sampling}\label{sec_mea}

In this section, we reinterpret Theorem \ref{theo1} in terms of the Brownian meander and the three dimensional Bessel process.

\subsection{The Brownian meander}

We first turn to the translation of Theorem \ref{theo1} in terms of the Brownian meander, denoted by $\big(m(u),~u\leq 1\big)$. To do so, we use an equality in law shown by Biane and Yor in \cite{biane1988quelques}. More precisely we have
$$\big((m(u),i_u),~u\leq 1\big)\underset{\mathcal{L}}{=}\big((|b(u)|+\lambda_u^0,\lambda_u^0),~u\leq 1\big),$$
where $$i_u=\underset{u\leq t\leq 1}{\text{inf}} m_t.$$ Thus, we can reinterpret Theorem \ref{theo2} as follows.

\begin{theorem}\label{theo3}
For any non negative measurable functions $f$ and $g$, we have
\begin{equation*}
\E\big[f\big(m(U),m(1)\big)g\big(\frac{i_U}{m(1)}\big)\big]=\sqrt{\frac{\pi}{2}}\E\big[f(\frac{1}{2}|B_1|+\Lambda L_1,L_1)L_1g(\Lambda)\big],\end{equation*}
with $\Lambda$ a uniform random variable on $[0,1]$, independent of $(B_1,L_1)$.
\end{theorem}
\noindent Let $(\tilde{\lambda}_1^y,~y\geq 0)$ denotes the family of local times of $m$ at time $1$. Similarly to what we have done for the Brownian bridge, we are able to retrieve from Theorem \ref{theo3} simple expressions for the laws of $m(1)$ and $m(U)$. We state these results in the following corollary.
\begin{corollary}\label{cortheo3}
The following properties hold:\\

\noindent $\bullet~$The variable $m(1)$ is Rayleigh distributed.\\

\noindent $\bullet~$The density of $m(U)$ at point $y\geq 0$ is given by
$$\E[\tilde{\lambda}_1^y]=2\int_y^{2y}\emph{exp}(-z^2/2)dz.$$
Thus, we have $m(U)\underset{\mathcal{L}}{=}\sqrt{2\mathcal{E}}W$, with $\mathcal{E}$ an exponential variable independent of $W$ which is uniformly distributed on $[1/2,1]$.
\end{corollary}
\begin{proof}
The proof of the first part of Corollary \ref{cortheo3} is obvious from Theorem \ref{theo3}. We now consider the second part. Let $f$ be a non negative measurable function. Using Theorem \ref{theo3} together with Equation \eqref{dens}, we get that 
$$\E\big[f\big(m(U)\big)\big]=\int_0^{+\infty}dx\int_0^{+\infty}dl l(x+l)\text{e}^{-(x+l)^2/2}\E\big[f(\frac{x}{2}+\Lambda l)\big].$$
Now remark that
$$\E\big[f(\frac{x}{2}+\Lambda l)\big]=\frac{1}{l}\int_{x/2}^{x/2+l}d\nu f(\nu).$$
Therefore, by Fubini's theorem, we get
\begin{align*}
\E\big[f\big(m(U)\big)\big]&=\int_0^{+\infty}dx\int_0^{+\infty}dl (x+l)\text{e}^{-(x+l)^2/2}\int_{x/2}^{x/2+l}d\nu f(\nu)\\
&=\int d\nu f(\nu)\int_0^{2\nu}dx\int_{\nu-x/2}^{+\infty}dl (x+l)\text{e}^{-(x+l)^2/2}.
\end{align*}
Thus, the density of $m(U)$ at point $\nu$ is given by
\begin{align*}
\int_0^{2\nu}dx\int_{\nu-x/2}^{+\infty}dl (x+l)\text{e}^{-(x+l)^2/2}&=\int_0^{2\nu}dx\text{exp}\big(-\frac{(x/2+\nu)^2}{2}\big)\\
&=2\int_{\nu}^{2\nu}dz\text{e}^{-z^2/2}.
\end{align*}
This ends the proof.
\end{proof}

\noindent In fact, as it is the case for the Brownian bridge, we can give explicit formulas for several other quantities related to the Brownian meander, for example the law of $m(U)$ given $m(1)$. However, these expressions are not so simple and therefore probably less interesting than those obtained for the Brownian bridge.

\subsection{The three dimensional Bessel process}

Finally, let $(R_t,~t\geq 0)$ be a three dimensional Bessel process starting from $0$ and
$$J_u=\underset{u\leq t\leq 1}{\text{inf}}R_t.$$
Using Imhof's absolute continuity relationship between the law of the meander and that of the three dimensional Bessel process, see \cite{biane1987processus,imhof1984density}, we may rewrite Theorem \ref{theo3} as follows.

\begin{theorem}\label{theo4}
For any non negative measurable functions $f$ and $g$, we have
\begin{equation*}
\E\big[f\big(R(U),R(1)\big)g\big(\frac{J_U}{R(1)}\big)\big]=\E\big[f(\frac{1}{2}|B_1|+\Lambda L_1,L_1)L_1^2g(\Lambda)\big],\end{equation*}
with $\Lambda$ a uniform random variable on $[0,1]$, independent of $(B_1,L_1)$.
\end{theorem}
\noindent We finally give the following corollary.

\begin{corollary}\label{cortheo4}
The following properties hold:\\

\noindent $\bullet~$The density of $R(1)$ at point $y\geq 0$ is given by
$$\sqrt{\frac{2}{\pi}}y^2\emph{exp}(-y^2/2).$$

\noindent $\bullet~$ $R(U)\underset{\mathcal{L}}{=}\sqrt{U}R(1)$ and its density at point $y\geq 0$ is given by
$$2\sqrt{\frac{2}{\pi}}y\int_y^{+\infty}\emph{exp}(-z^2/2)dz.$$

\noindent $\bullet~$The law of $R(U)$ given $R(1)$ is the same as the law of $m(U)$ given $m(1)$.

\end{corollary}
\noindent The first two parts of Corollary \ref{cortheo4} are in fact easily deduced from basic properties of the three dimensional Bessel process. The last part is a consequence of Imhof's relation.
 
\section{The centering property of $\alpha$ revisited through an enlargement of filtration formula}\label{sec_filt}
In this last section, we revisit the centering property of the variable
$$\alpha=\frac{B_{UT_1}}{\sqrt{T_1}},$$ which is proved in \cite{elie2013expectation} and leads to various developments in \cite{rosenbaum2013law}.
Our goal here is to show that this result can be recovered from simple properties of the three dimensional Bessel process sampled at uniform time, together with an enlargement of filtration formula for the Brownian motion with the time $T_1$ due to Jeulin, see \cite{jeulin1979grossissement}. 

\subsection{Some preliminary remarks on the uniformly sampled Bessel process}
Let $U$ be a uniform random variable on $[0,1]$, independent of the considered Bessel process $R$. We start with the two following lemmas on the conditional expectation of the uniformly sampled Bessel process.

\begin{lemma}\label{lem1}
We have
\begin{equation}\label{lem1_1}
\E[R_U|R_1=r]=\frac{1}{2}\big(r+\E[\frac{U}{R_U}|R_1=r]\big).
\end{equation}
Consequently,
$$\E\big[\frac{R_U}{R_1^2}\big]=\frac{1}{2}\big(\sqrt{\frac{2}{\pi}}+\E[\frac{U}{R_UR_1^2}]\big).$$
\end{lemma}

\begin{lemma}\label{lem2}
We have
$$\E[\frac{U}{R_U}|R_1=r]=H(r).$$
Consequently,
$$\E\big[\frac{R_U}{R_1^2}\big]=\E[\frac{U}{R_UR_1^2}]=\sqrt{\frac{2}{\pi}}.$$
\end{lemma}

\noindent Remark that we already proved the equality 
$$\E\big[\frac{R_U}{R_1^2}\big]=\sqrt{\frac{2}{\pi}}$$
in \cite{elie2013expectation}. This was in fact the cornerstone of our first proof of the centering property of $\alpha$. In the enlargement of filtration approach used here, we will see that instead of $R_U/R_1^2$, the random variable $U/(R_UR_1^2)$ appears naturally.

\subsection{Proofs of Lemma \ref{lem1} and Lemma \ref{lem2}}

We now give the proofs of Lemma \ref{lem1} and Lemma \ref{lem2}.

\subsubsection{Proof of Lemma \ref{lem1}}

The first part of Lemma \ref{lem1} follows from the identity
\begin{equation}\label{returnbes}
\E[\frac{R_t}{t}|R_1]=R_1+\E[\int_t^1\frac{dv}{vR_v}|R_1],~t\leq 1,
\end{equation}
after multiplying both sides by $t$ and integrating in $t$ from $0$ to $1$. To show \eqref{returnbes}, we use time inversion with $t=1/w$ and $$R'_w=wR_{1/w},$$ another three dimensional Bessel process. With this notation, using the Ito representation of the Bessel process, we get
$$\E[R'_w-R'_1|R'_1]=\E[\int_1^w\frac{dt}{R'_t}|R'_1],$$ from which \eqref{returnbes} is easily obtained. The second statement in Lemma \ref{lem1} readily follows. 

\subsubsection{Proof of Lemma \ref{lem2}}
We start with the proof of the first part of Lemma \ref{lem2}. Let
$$\rho=\E[\int_0^1dv\frac{v}{R_v}|R_1=r].$$
Using the same time inversion trick as for the proof of Lemma \ref{lem1} together with the Markov property, we get
$$\rho=\int_1^{+\infty}\frac{dw}{w^2}\E[\frac{1}{R_w'}|R_1'=r]=\int_0^{+\infty}\frac{dt}{(1+t)^2}\E_r[\frac{1}{R'_t}],$$
where $\mathbb{P}_r$ denotes the law of a Bessel process $R'$ starting from $r$. We then use the Doob's absolute continuity relationship, that is
$$\mathbb{P}_r\big|_{\mathcal{F}_t}=\frac{X_{t\wedge T_0}}{r}W_r\big|_{\mathcal{F}_t},$$
where $W_r$ is the Wiener measure associated to a Brownian motion starting at point $r$, $X$ is the canonical process and $T_0$
is the first hitting time of $0$ by $X$, see for example \cite{revuz1999continuous}, Chapter XI. This together with the fact that
$$\frac{X_{t\wedge T_0}}{X_t}=\mathrm{1}_{T_0>t}$$
gives
$$\E_r[\frac{1}{R'_t}]=\frac{1}{r}W_r[T_0>t]=\frac{1}{r}W_0[T_r>t].$$
Therefore,
$$\rho=\frac{1}{r}\E^{W_0}\big[\int_0^{T_r}\frac{dt}{(1+t)^2}\big]=\frac{1}{r}\E^{W_0}[\frac{T_r}{1+T_r}]=r\E[\frac{1}{N^2+r^2}].$$
Using Equation \eqref{double2}, this is equal to $H(r)$. This ends the proof of the first part of Lemma \ref{lem2}. Using the expression for the density of $R_1$ given in Corollary \ref{cortheo4}, the proof of the second part readily follows remarking that
$$\E\big[\frac{U}{R_UR_1^2}\big]=\sqrt{\frac{2}{\pi}}\int_{0}^{+\infty}dr\int_{r}^{+\infty}dx\text{e}^{-x^2/2}=\sqrt{\frac{2}{\pi}}\int_{0}^{+\infty}dx x\text{e}^{-x^2/2}=\sqrt{\frac{2}{\pi}}.$$

\subsection{An enlargement of filtration approach to the centering property of $\alpha$}
We now revisit the centering property of $\alpha$ through an enlargement of filtration formula. Let $(\mathcal{F}_t)$ denote the filtration of the Brownian motion $(B_t)$ and $(\mathcal{F}'_t)$ the filtration obtained by initially enlarging $(\mathcal{F}_t)$ with $T_1$. It is shown in \cite{jeulin1979grossissement} that $(B_t)$ is a $(\mathcal{F}'_t)$ semi-martingale. More precisely,
\begin{equation}\label{dec}
B_t=\beta_t-\int_0^{t\wedge T_1}\frac{ds}{1-B_s}+\int_0^{t\wedge T_1}ds\frac{1-B_s}{T_1-s},
\end{equation}
where $(\beta_t)$ is a $(\mathcal{F}'_t)$ Brownian motion (in particular it is independent of $T_1$).
Taking expectation on both sides of \eqref{dec} at time $t=UT_1$, we get
$$\E[\alpha]=-\E\big[\frac{1}{\sqrt{T_1}}\int_0^{UT_1}\frac{ds}{1-B_s}\big]+\E\big[\frac{1}{\sqrt{T_1}}\int_0^{UT_1}ds\frac{1-B_s}{T_1-s}\big].$$
Using the change of variable $s=uT_1$ in both integrals, we get
$$\E[\alpha]=-\E\big[\sqrt{T_1}\int_0^{U}\frac{du}{1-B_{uT_1}}\big]+\E\big[\frac{1}{\sqrt{T_1}}\int_0^{U}du\frac{1-B_{uT_1}}{1-u}\big].$$
Since $U$is independent of $B$ and uniformly distributed on $[0,1]$, we get
$$\E[\alpha]=-\E\big[\sqrt{T_1}\int_0^{1}du\frac{(1-u)}{1-B_{uT_1}}\big]+\E\big[\frac{1}{\sqrt{T_1}}\int_0^{1}du(1-B_{uT_1})\big].$$
Thus,
$$2\E[\alpha]=-\E\big[\sqrt{T_1}\int_0^{1}dv\frac{v}{1-B_{T_1(1-v)}}\big]+\E\big[\frac{1}{\sqrt{T_1}}\big].$$
We now use Williams time reversal result:
$$\Big(T_1,\big(1-B_{T_1(1-v)},~v\leq 1\big)\Big)\underset{\mathcal{L}}{=}\Big(\gamma_1,\big(R_{v\gamma_1},~v\leq 1\big)\Big),$$
where
$$\gamma_1=\text{sup}\{s>0,~R_s=1\}.$$
Hence we obtain
$$2\E[\alpha]=-\E\big[\frac{V}{R_{V\gamma_1}/\sqrt{\gamma_1}}\big]+\E\big[\frac{1}{\sqrt{T_1}}\big],$$ with $V$ a uniform random variable on $[0,1]$, independent of $R$. From the absolute continuity relationship between the laws of 
$$\big(R_{v\gamma_1}/\sqrt{\gamma_1},~v\leq 1\big)$$ and $(R_{v},~v\leq 1),$ see \cite{biane1987processus}, we get
$$2\E[\alpha]=\sqrt{\frac{2}{\pi}}-\E[\frac{V}{R_VR_1^2}].$$
Hence $\E[\alpha]=0$ if and only if
$$E[\frac{V}{R_VR_1^2}]=\sqrt{\frac{2}{\pi}}.$$ From Lemma \ref{lem2}, the last equality holds. Moreover, it has been shown without the help of our previous results \cite{elie2013expectation,rosenbaum2013law}. Thus, the use of the enlargement formula of \cite{jeulin1979grossissement} provides an alternative proof of the centering property of $\alpha$.

\section{A few words of conclusion}

Together with \cite{elie2013expectation} and \cite{rosenbaum2013law}, this paper is our third work where various aspects of the law of 
\begin{equation*}
\alpha=\frac{B_{UT_1}}{\sqrt{T_1}}
\end{equation*} are investigated. 
For example, we have considered its centering property, the explicit form of its density, which may be directly deduced from Equation \eqref{law} and Equation \eqref{dens}, and its Mellin transform. In the present paper, starting from the pseudo-Brownian bridge, we obtain some results relative to the Brownian bridge, the Brownian meander and the three dimensional Bessel process. Imhof type relations between these processes allow to go from one to another.

\bibliographystyle{abbrv}
\bibliography{bibli_ry2}
\end{document}